\documentclass[11pt]{amsart}

\newcommand{\private}[1]{}
\newcommand{\bfn}[1]{}                          
\newcommand{\ifn}[1]{}      
        
\usepackage{amssymb, amsmath, amscd, amsthm, color, epsfig,url}

\usepackage[all]{xy}          
\xyoption{dvips}              

\makeatletter
\renewcommand\l@subsection{\@tocline{2}{0pt}{2pc}{5pc}{}}
\makeatother

\newcommand{\R}{{\mathbb R}}

\newcommand{\abs}[1]{{\left\vert #1 \right\vert}}

\newcommand{\hofiber}{\operatorname{hofiber}}

\newcommand{\holim}{\operatorname{holim}}

\newcommand{\tfiber}{\operatorname{tfiber}}
\newcommand{\tcofiber}{\operatorname{tcofiber}}
\newcommand{\Map}{\operatorname{Map}}

\newcommand{\Imm}{\operatorname{Imm}}
\newcommand{\Link}{\operatorname{Link}}

\newcommand{\del}{{\partial}}

\newcommand{\Top}{\operatorname{Top}}

\theoremstyle{plain}
\newtheorem{thm}{Theorem}[section]
\newtheorem{prop}[thm]{Proposition}
\newtheorem{lemma}[thm]{Lemma}
\newtheorem{cor}[thm]{Corollary}

\theoremstyle{definition}
\newtheorem{defin}[thm]{Definition}
\newtheorem{example}[thm]{Example}
\newtheorem{def/ex}[thm]{Definition/Example}

\theoremstyle{remark}
\newtheorem{rem}[thm]{Remark}

\newcommand{\refS}[1]{Section~\ref{S:#1}}
\newcommand{\refT}[1]{Theorem~\ref{T:#1}}
\newcommand{\refC}[1]{Corollary~\ref{C:#1}}
\newcommand{\refP}[1]{Proposition~\ref{P:#1}}
\newcommand{\refD}[1]{Definition~\ref{D:#1}}
\newcommand{\refL}[1]{Lemma~\ref{L:#1}}
\newcommand{\refE}[1]{equation~$(\ref{E:#1})$}

\begin{document}


\title[Derivatives of the identity and generalizations of Milnor's invariants]{Derivatives of the identity and generalizations of Milnor's invariants}


\author{Brian A. Munson}
\address{Department of Mathematics, Wellesley College, Wellesley, MA}
\email{bmunson@wellesley.edu}
\urladdr{http://palmer.wellesley.edu/\~{}munson}



\begin{abstract}
{\bf Version: \today}
We synthesize work of U. Koschorke on link maps and work of B. Johnson on the derivatives of the identity functor in homotopy theory. The result can be viewed in two ways: (1) As a generalization of Koschorke's ``higher Hopf invariants'', which themselves can be viewed as a generalization of Milnor's invariants of link maps in Euclidean space; and (2) As a stable range description, in terms of bordism, of the cross effects of the identity functor in homotopy theory evaluated at spheres. We also show how our generalized Milnor invariants fit into the framework of a multivariable manifold calculus of functors, as developed by the author and Voli\'{c}, which is itself a generalization of the single variable version due to Weiss and Goodwillie. 
\end{abstract}

\maketitle

\tableofcontents

\parskip=4pt
\parindent=0cm


\section{Introduction}\label{S:Intro}


The main results of this paper are a synthesis, reinterpretation, and generalization of work of Johnson \cite{J:DHT} and Koschorke \cite{Kosch:Milnor}. It is also, in a sense, a continuation of work done by the author in \cite{M:LinkNumber}. The main result which expresses this is the following.

\begin{thm}\label{T:maintheorem}
Let $P$ be a smooth closed manifold of dimension $p$. There is a $((k+1)(n-2)-p)$-connected map, the ``total higher Hopf invariant of order $k$''
\begin{equation}\label{E:mainthmeqn}
H_k:\Map_\ast\left(P_+,\tfiber\left(S\mapsto \vee_{\underline{k}-S} S^{n-1}\right)\right)\rightarrow \prod_{(k-1)!}Q_+T\left(P;\epsilon^{k(n-2)+1}-TP\right),
\end{equation}

where $S$ ranges through subsets of $\underline{k}=\{1,\ldots, k\}$.
\end{thm}
The domain of $H_k$ is the space of maps from $P$ to the $k^{th}$ cross-effect of the identity functor (the ``total homotopy fiber'' of a $k$-cube of wedges of spheres) evaluated at $S^{n-1}$. The codomain is the infinite loopspace associated to the Thom spectrum of a virtual vector bundle, in this case $k(n-2)+1$ copies of the trivial bundle minus the tangent bundle of $P$. For more details about total homotopy fibers, see \refS{Back}; the notation in the codomain is explained in \refS{Conv}. 

Let us explain in a little more detail how this theorem relates the work of Johnson and Koschorke. Let $X_1,\ldots, X_k$ be based spaces. Johnson's work \cite{J:DHT} on the derivatives of the identity functor involved producing a highly-connected map from the loopspace of $\tfiber\left(S\mapsto\vee_{i\in \underline{k}-S}\Sigma X_i\right)$ to the loopspace of $\Map_\ast\left(\Delta_k,\Sigma X_1\wedge \Sigma X_2\wedge\cdots\wedge \Sigma X_k\right)$ which identifies the homotopy type of the derivatives of the identity functor. Suffice it to say here that $\Delta_k$ is a based complex with the homotopy type of $\vee_{(k-1)!}S^{k-1}$; see \refS{thespacedeltak} for more details. This map induces the map $H_k$ in \refT{maintheorem} if we set all $X_i$ to be spheres of the same dimension, essentially by composing with $\Map_\ast(P_+,-)$. 

There are two observations about the map $H_k$ which are central to this work: (1) the domain of $H_k$ arises as a homogeneous layer in a multivariable Taylor tower for the space of link maps when $P$ is a $(k+1)$-fold product; and (2) the codomain of $H_k$ has a bordism interpretation as a home for higher-order linking numbers, which follows from ideas due to Koschorke \cite{Kosch:Milnor}. A few more words about both of these are in order. Let us deal with (2) first. 

The codomain of $H_k$ is a homotopy theoretic model for a space whose points may be thought of as the ``higher-order linking numbers'' of the bordism class of a link $L=L_1\coprod \cdots\coprod L_k$ over the manifold $P$. Here is a technically incorrect, but suitable for the purposes of the introduction, way to construct the link in question (see \refS{InvtsMflds} for the correct version). Let $x_i\in S^{n-1}$ be a non-wedge point in the $i^{th}$ copy of $S^{n-1}$ in $\vee_{\underline{k}}S^{n-1}$. Given a map $f\in\Map_\ast\left(P_+,\tfiber\left(S\mapsto \vee_{\underline{k}-S}S^{n-1}\right)\right)$, a small homotopy will make it transverse to $\coprod_i x_i$, and denote by $L_i$ the inverse image of $x_i$ by $f$. This produces a link $L=\coprod_i L_i$ over $P$. Really it is a bordism class of a link, because it is only defined up to a homotopy of the map from $P$ to the wedge of spheres, but we will ignore such details for the time being. The link $L$ is an honest link embedded in $P$, but we prefer to think of it as a framed manifold $L=\coprod_i L_i$ together with a \emph{link map} to $X$; that is, a map to $X$ for which the image of $L_i$ and $L_j$ are disjoint for all $i\neq j$. From the link $L$ we can make ``higher-order linking numbers'' (which are themselves bordism classes of manifolds) in the sense of Koschorke \cite{Kosch:Milnor}, who generalized Milnor's invariants of classical links to linking of spheres of arbitrary dimension in Euclidean space. It is not clear from this sketchy description that these higher-order linking numbers should be invariants at all, but we have a different and more direct way of defining it so that this is obvious. The information encoded in $\tfiber\left(S\mapsto \vee_{\underline{k}-S}S^{n-1}\right)$ gives a reason why certain lower-order invariants vanish. It turns out that these higher-order linking numbers can be understood as an ``overcrossing locus'', an observation also made by Koschorke. This is discussed in \refS{MilnorInvts}, along with our \refC{kosthm1gen}, which is a generalization of Koschorke's theorem expressing the close relationship between Whitehead products and his higher Hopf invariants. 

As for (1), when $P=P_1\times\cdots\times P_{k+1}$, the domain of the map $H_k$ naturally arises in trying to understand the higher-order relative linking numbers of a pair of link maps $e,f:P_1\coprod\cdots\coprod P_{k+1}\to\R^n$ (they must satisfy a relative analog of ``almost triviality'' described in \refS{genmilnor}). Let $\Link\left(P_1,\ldots, P_{k+1};\R^n\right)$ denote the space of smooth maps $\coprod_iP_i\to\R^n$ such that the images of $P_i$ and $P_j$ are disjoint for all $i\neq j$. The space $\Map_\ast\left(\left(P_1\times\cdots\times P_{k+1}\right)_+,\tfiber\left(S\mapsto \vee_{\underline{k}-S} S^{n-1}\right)\right)$ appears as the ``homogeneous degree $(1,1,\ldots, 1)$ layer'' of the multivariable Taylor tower of the space of link maps, giving us the following corollary. Write $\vec{1}=(1,1,\ldots, 1)$.

\begin{cor}\label{C:cormainthm}
Let $P_1,\ldots, P_{k+1}$ be smooth closed manifolds of dimensions $p_1,\ldots, p_{k+1}$ respectively, and let $p=\sum_ip_i$. There is a $((k+1)(n-2)-p)$-connected map
\begin{equation*}
L_{\vec{1}}\Link\left(P_1,\ldots, P_{k+1};\R^n\right)\rightarrow \prod_{(k-1)!}QT\left(P_1\times\cdots\times P_{k+1};\epsilon^{k(n-2)+1}-T\left(P_1\times\cdots\times P_{k+1}\right)\right).
\end{equation*}
\end{cor}

The multivariable manifold calculus of functors is a generalization due to the author and Voli\'{c} \cite{MV:Multi} of the manifold calculus of functors due to Weiss and Goodwillie \cite{W:EI1,GW:EI2}, built with the space of link maps in mind. What \refC{cormainthm} tells us is that there is a stable range description, in terms of bordism, of certain layers in the multivariable manifold calculus tower of link maps, and that this stable range description, that is, the map $H_k$, admits a geometric interpretation as a generalization of Koschorke and Milnor's invariants for link maps. This is covered in greater detail in \refS{Applications}, where we also compute some mapping space models for $L_{\vec{1}}\Link\left(P_1,\ldots, P_{k+1};N\right)$ which involve maps of products of the $P_i$ into configuration spaces. These models mirror how Koschorke \cite{Kosch:Milnor} built his invariants of a $(k+1)$-component link map $f=\coprod_if_i:\coprod_i S^{p_i}\to \R^n$ from $\widehat{f}=\prod_if_i:\prod_iS^{p_i}\to C(k+1,\R^n)$, where the codomain of $\widehat{f}$ is the configuration space of $k+1$ points in $\R^n$.

Bruce Williams and John Klein have independently discovered an invariant similar to our $H_k$. Here is the setup: Suppose $P_1,\ldots, P_k$ are embedded disjointly in a smooth manifold $N$. They study the problem of making a map $P_{k+1}\to N$ disjoint from the $P_i$. For a subset $S\subset \{1,\ldots, k\}$, let $P_S=\cup_{i\in\underline{k}-S} P_i$, and consider the diagram $S\mapsto\Map\left(P_{k+1},N-P_S\right)$. There is a map $\hofiber\left(\Map\left(P_{k+1},N-P_{\underline{k}})\to\holim_{S\neq\emptyset}\Map(P_{k+1},N-P_S\right)\right)$ to a space similar to the codomain of $H_k$ (and equal to it in the case where $N=\R^n$), and they can show that it has a high connectivity, using different methods than we employ. Their methods are not available to us because we do not assume our link maps $\coprod_iP_i\to N$ are embeddings (Klein and Williams require all but $P_{k+1}$ be be embedded).

A natural consequence of the observations discussed above is a generalization of Koschorke's higher-order linking numbers \cite{Kosch:Milnor}. Our generalization builds upon Koschorke's work in the following ways:
\begin{enumerate}
\item We have a notion of higher-order linking of arbitrary smooth closed manifolds $P_i$; Koschorke's work focuses on spheres.
\item All of our constructions are functorial in the $P_i$.
\item All of our constructions are on the level of spaces instead of groups, hence the invariance (under the suitable notion of equivalence) of all of our constructions is immediate.
\item We eliminate the assumption that a link map be ``$\kappa$-Brunnian'', and replace it with a relative analog of almost triviality (see \refS{Applications} for definitions). That is, instead of trying to measure whether a given link map is homotopic to the trivial link, we instead focus on the question of whether two arbitrary link maps $e,f$ are link-homotopic. We replace the assumption that a link is almost trivial with a condition very close to assuming all of the sub-link maps of $e$ and $f$ are link-homotopic (ours is implied by a slightly stronger condition than this; see \refS{genmilnor}). This is analogous to the fact that the linking number should be thought of a relative invariant of link maps, a point of view explored in \cite{M:LinkNumber}.
\end{enumerate}

Finally, we wish to reiterate that this work lies at the intersection of two different versions of calculus of functors intended for different purposes: the homotopy calculus, as evidenced by our use of Johnson's work on the derivatives of the identity, and manifold calculus applied to the space of link maps.

This paper is organized as follows. \refS{Back} contains some background material concerning cubical diagrams and total homotopy fibers. In \refS{JohnsonThesis} we discuss the main theorem of \cite{J:DHT}, some of its corollaries, and the details of its proof we will need later on. \refS{MilnorInvts} discusses Koschorke's work \cite{Kosch:Milnor} on higher Hopf invariants and generalizations of Milnor's invariants, and we prove a generalization, \refT{kosthm1gengen} of his main theorem (appearing here as \refT{kosthm1}). \refS{LinkBordism} is a detour into ``cobordism spaces'' which are used for our bordism description of the map in \refT{maintheorem} and its geometric interpretation as an overcrossing locus, which also appears in this section. Finally, \refS{Applications} applies our theorem to the study of link maps in Euclidean space. We show how these invariants are organized by a multivariable manifold calculus tower, and we give explicit models for the relevant stages of this tower as homotopy limits of diagrams of maps into configuration spaces.

It would be interesting to know what invariants the other homogeneous layers of the Taylor tower for link maps contain, and under what conditions, if any, the multivariable Taylor tower converges to the space of link maps. The latter appears to be a very difficult question, and the former a potentially very interesting one, especially in the classical case of link maps of circles in $\R^3$. It would also be interesting to extend the results of this paper to link maps into a generic manifold $N$.


\subsection{Conventions}\label{S:Conv}

For a finite set $S$, let $|S|$ stand for its cardinality. Let $\underline{k}$ denote the set $\{1,\ldots, k\}$. For a space $X$ equipped with a vector bundle $\xi$, let $T(X;\xi)$ denote the Thom space. It is the quotient of the total space $D(X;\xi)$ of this bundle by the sphere bundle $S(X;\xi)$. For a based space $X$, we let $QX$ stand for $\Omega^{\infty}\Sigma^{\infty} X$. If $X$ is unbased, we let $X_+$ denote $X$ with a disjoint basepoint added and let $Q_+X$ stand for $Q(X_+)$. Basepoints will be denoted $\ast$ unless otherwise noted.


\section{Background}\label{S:Back}


We begin this section with a discussion of cubical diagrams. In particular we will review a detailed description of the ``total homotopy fiber'' of a cubical diagram. We also include a result which allows us to compare the connectivities of certain adjoint maps.


\subsection{Cubical diagrams}\label{S:Cubes}


We will assume the reader is familiar with homotopy limits.  Details about cubical diagrams can be found in \cite[Section 1]{CalcII}. For a finite set $T$, let $\mathcal{P}(T)$ be the poset of subsets of $T$, and $\mathcal{P}_0(T)$ the sub-poset of nonemtpy subsets.

\begin{defin}
Let $T$ be a finite set. A \textsl{$\abs{T}$-cube} of spaces is a covariant functor $$\mathcal{X}\colon\mathcal{P}(T)\longrightarrow\Top.$$ We may also speak of a cube of based spaces; in this case, the target is $\Top_\ast$.
\end{defin}

\begin{defin}
The \textsl{total homotopy fiber}, or \textsl{total fiber}, of a $\abs{T}$-cube $\mathcal{X}$ of based spaces, denoted $\tfiber(S\mapsto \mathcal{X}(S))$ or $\tfiber(\mathcal{X})$, is the homotopy fiber of the map 
$$
\mathcal{X}(\emptyset)\longrightarrow\underset{S\in\mathcal{P}_0(T)}{\holim}\, \, \mathcal{X}(S).
$$ 
If this map is a $k$-connected, we say the cube is \emph{$k$-cartesian}. In case $k=\infty$, (that is, if the map is a weak equivalence), we say the cube $\mathcal{X}$ is \textsl{homotopy cartesian}.
\end{defin}

The total fiber can also be thought of as an iterated homotopy fiber. That is, view a $\abs{T}$-cube $\mathcal{X}$ as a a natural transformation of $(\abs{T}-1)$-cubes $\mathcal{Y}\rightarrow\mathcal{Z}$. In this case, $\tfiber(\mathcal{X})=\hofiber(\tfiber(\mathcal{Y})\rightarrow\tfiber(\mathcal{Z}))$.

An equivalent, and more descriptive, definition of total homotopy fiber is the following.

\begin{defin}\label{D:tfiber}[Definition 1.1 of \cite{CalcII}]
Let $\mathcal{X}$ be a $|T|$-cube of based spaces. A point $\Phi\in\tfiber(\mathcal{X})$ is a collection of maps $\Phi_S:I^S\to \mathcal{X}(S)$ for each subset $S\subset T$ satisfying the following two conditions.
\begin{enumerate}
\item $\Phi$ is natural with respect to $S$: let $S'\subset S\subset T$ and let $I^{S'}\to I^{S}$ be the inclusion by zero on the coordinates of $I^S$ labeled by $S-S'$. Then the following diagram commutes:
$$
\xymatrix{
I^{S'}\ar[r]\ar[d]_{\Phi_{S'}} & I^S\ar[d]^{\Phi_S}\\
\mathcal{X}(S')\ar[r] & \mathcal{X}(S)\\
}
$$
\item For each $S$, $\Phi_S$ takes $(I^S)_1=\{u\in I^S | \mbox{ there exists $s\in S$ such that }u_s=1\}$ to the basepoint.
\end{enumerate}
\end{defin}

We end with a proposition that will be useful in the proof of \refT{kosthm1gengen}.

\begin{prop}\label{P:connadjoints}
Suppose $X$ and $Y$ are based spaces and that $X$ is $n$-connected. Let $f:X\to \Omega Y$ and $\widetilde{f}:\Sigma X\to Y$ adjoint maps. Then
\begin{itemize}
\item If $f$ is $k$-connected, $\widetilde{f}$ is $(\min\{2n+2,k\}+1)$-connected. 
\item If $\widetilde{f}$ is $(k+1)$-connected, $f$ is $\min\{2n+1,k\}$-connected.
\end{itemize}
\end{prop}

\begin{proof}
Consider the commutative diagram
$$
\xymatrix{
X\ar[r]\ar[dr]_{f} & \Omega\Sigma X \ar[d]^{\Omega \widetilde{f}}\\
&\Omega Y
}
$$
where $X\to\Omega\Sigma X$ is the canonical map. Since $X$ is $n$-connected, $X\to\Omega \Sigma X$ is $(2n+1)$-connected by the Freudenthal Suspension Theorem. If $\widetilde{f}$ is $(k+1)$-connected, then $\Omega \widetilde{f}$ is $k$-connected, and hence $f$ is $\min\{2n+1,k\}$-connected. If $f$ is $k$-connected, then $\Omega \widetilde{f}$ is $\min\{2n+2,k\}$-connected, and hence $\widetilde{f}$ is $(\min\{2n+2,k\}+1)$-connected.
\end{proof}


\section{The derivatives of the identity functor}\label{S:JohnsonThesis}


The work of Johnson \cite{J:DHT} is concerned with computing an explicit description of the derivatives of the identity functor in homotopy calculus. One important aspect of this description is a space $\Delta_k$, mentioned in the introduction, which is a quotient of the product of $k$ copies of the $(k-1)$-cube by certain subspaces. It has the homotopy type of $\vee_{(k-1)!}S^{k-1}$. Johnson proves the following theorem.

\begin{thm}\cite[Theorem 2.2]{J:DHT}\label{T:JohnsonThesis}
Let $X_1,\ldots, X_k$ be based spaces. There is a natural transformation of functors of $X_1,\ldots, X_k$
$$
T_k:\tfiber(S\mapsto\vee_{i\in \underline{k}-S}X_i)\longrightarrow\Map_\ast(\Delta_k,X_1\wedge X_2\wedge\cdots\wedge X_k)
$$
such that if $X_1,\ldots, X_k$ are $n$-connected, then the map
$$
\Omega T_k:\Omega\tfiber(S\mapsto\vee_{i\in \underline{k}-S}\Sigma X_i)\longrightarrow\Omega\Map_\ast(\Delta_k,\Sigma X_1\wedge \Sigma X_2\wedge\cdots\wedge \Sigma X_k)
$$
is $((k+1)(n+1)-1)$-connected.
\end{thm}

\begin{rem}
The cube $S\mapsto\vee_{i\in \underline{k}-S}\Sigma X_i$ is $(k(n+1)+1)$-cartesian if the $X_i$ are $n$-connected by the Blakers-Massey Theorem (\cite{ES:BM}, also see \cite[Theorem 2.3]{CalcII}); easier still to see is that the codomain of $T_k$ is $k(n+1)$-connected. In any case, we can interpret the second part of \refT{JohnsonThesis} as a stable range description of the homotopy type of a ``stabilization'' of $\tfiber(S\mapsto\vee_{i\in \underline{k}-S} X_i)$.
\end{rem}


\subsection{Corollaries of \refT{JohnsonThesis}}\label{S:CorJohnsonThesis}


The following corollary is important for our main result, and follows immediately. For these corollaries, we need that the spaces $X_i$ are connected.

\begin{cor}\label{C:unloopedtk}
If $X_1,\ldots, X_k$ are $n$-connected spaces with $n\geq 0$, then
$$
T_k:\tfiber(S\mapsto\vee_{i\in \underline{k}-S}\Sigma X_i)\longrightarrow\Map_\ast(\Delta_k,\Sigma X_1\wedge\Sigma X_2\wedge\cdots\wedge\Sigma X_k)
$$
is $(k+1)(n+1)$-connected.
\end{cor}

As we mentioned above, $\Delta_k$ is homotopy equivalent to $\vee_{(k-1)!}S^{k-1}$, a fact we will discuss in more detail in \refS{JohnsonRemarks}. Thus we have an equivalence 
$$
\Map_\ast(\Delta_k,\Sigma X_1\wedge\Sigma X_2\wedge\cdots\wedge\Sigma X_k)\simeq\prod_{(k-1)!}\Omega^{k-1}\Sigma^{k-1}\Sigma(\wedge_{i=1}^kX_i).\bfn{you'll want to say which $\Sigma$ that is, and I think it's the one that goes with $X_1$, since that's the guy which gets ``fixed'' later on}
$$
Let $T_k^s:\Map_\ast(\Delta_k,\Sigma X_1\wedge\Sigma X_2\wedge\cdots\wedge\Sigma X_k)\to\prod_{(k-1)!}Q\Sigma(\wedge_{i=1}^kX_i)$ denote the composition of $T_k$ with the canonical map $\prod_{(k-1)!}\Omega^{k-1}\Sigma^{k-1}\Sigma(\wedge_{i=1}^kX_i)\to\prod_{(k-1)!}Q\Sigma(\wedge_{i=1}^kX_i)$. The following corollary immediately follows from the Freudenthal Suspension Theorem.

\begin{cor}\label{C:stableunloopedtk}
$$
T_k^s:\tfiber(S\mapsto\vee_{i\in \underline{k}-S}\Sigma X_i)\longrightarrow \prod_{(k-1)!}Q\Sigma(\wedge_{i=1}^k X_i).
$$
is $(k+1)(n+1)$-connected.
\end{cor}

Let $P$ be a $p$-dimensional manifold (or CW complex). Composing with $\Map_\ast(P_+,-)$, we have the following.

\begin{cor}\label{C:stableunloopedtksphere}
$T^s_k$ induces a $((k+1)(n+1)-p)$-connected map
\begin{equation}\label{E:tfiberhtpyconn}
(T_k^s)_\ast:\Map_\ast(P_+,\tfiber(S\mapsto\vee_{i\in \underline{k}-S}\Sigma X_i))\longrightarrow \prod_{(k-1)!}\Map_\ast(P_+,Q\Sigma(\wedge_{i=1}^k X_i)).
\end{equation}
\end{cor}
The special case we are concerned with is when $X_i=S^{n-2}$ for all $i$. In that case we have a $((k+1)(n-2)-p)$-connected map
\begin{equation}\label{E:tfiberhtpyconnsphere}
(T_k^s)_\ast:\Map_\ast(P_+,\tfiber(S\mapsto\vee_{i\in \underline{k}-S}S^{n-1}))\longrightarrow \prod_{(k-1)!}\Map_\ast(P_+,QS^{k(n-2)+1}).
\end{equation}

\begin{rem}
Let $P$ be a point and $X_i=S^{n-2}$ for all $i$. When $k=1$, this says that $S^{n-1}\to QS^{n-1}$ is $(2n-4)$-connected, while the Freudenthal Suspension Theorem tells us the map is actually $(2n-3)$-connected. The reason for the discrepancy is that we are getting this connectivity estimate from \refC{unloopedtk}, which says in the case $k=1$ that the identity map $\Sigma S^{n-2}\to \Sigma S^{n-2}$ is $(2n-4)$-connected, far less than optimal. Using this we are only able to conclude that the composed map $\Sigma S^{n-2}\to Q\Sigma S^{n-2}$ is $(2n-4)$-connected.

Similarly, when $k=2$, \refC{stableunloopedtksphere} says that
$$
\hofiber(S^{n-1}\vee S^{n-1}\to S^{n-1}\times S^{n-1})\longrightarrow QS^{2n-3}
$$
is $(3n-6)$-connected, while it is in fact $(3n-5)$-connected. A geometric understanding of this map was crucial to the main theorem of \cite{M:Emb}. It is also easy to see how this is related to the linking number.  By the Hurewicz Theorem, to show this map is $(3n-5)$-connected, it suffices to show the map induces an isomorphism on $\pi_{2n-3}$. A generator for $\pi_{2n-3}\hofiber(S^{n-1}\vee S^{n-1}\to S^{n-1}\times S^{n-1})$ is given by the Whitehead product $\iota:S^{2n-3}\to \hofiber(S^{n-1}\vee S^{n-1}\to S^{n-1}\times S^{n-1})$ of the inclusion maps $S^{n-1}\to S^{n-1}\vee S^{n-1}$. One way to see that the Whitehead product is indeed a generator is to form the link $\iota^{-1}(p_1)\cup\iota^{-1}(p_2)$, and show that it has linking number $\pm1$. Here $p_1,p_2$ are non-wedge points, one from each of the spheres in question. See \cite[4.2.2]{M:Emb} for details. We are not certain whether the connectivity estimate can be improved by one in general.
\end{rem}

In \refS{LinkBordism} we will give a bordism interpretation of the target of the map in \refE{tfiberhtpyconnsphere}. Now we turn to an exploration of the objects and constructions used in the proof of \refT{JohnsonThesis} necessary for such a bordism interpretation.


\subsection{Remarks on the proof of \refT{JohnsonThesis}}\label{S:JohnsonRemarks}


We require a more thorough understanding of the space $\Delta_k$ and the map $T_k$ in \refT{JohnsonThesis}. Johnson constructs a map
$$
T_k':\tfiber(S\mapsto\vee_{i\in \underline{k}-S}X_i)\longrightarrow\Map_\ast(I^{k(k-1)},X_1\times X_2\times\cdots\times X_k)
$$
defined as follows. Recall from \refD{tfiber} that a point in $\Phi\in\tfiber(S\mapsto\vee_{i\in \underline{k}-S}X_i)$ consists in part of a collection of maps $\Phi_i:I^{k-1}\to X_i$ for $i=1$ to $k$ satisfying certain properties. She defines
\begin{equation}\label{E:tkprime}
T'_k(\Phi)=\prod_{i=1}^k\Phi_i:\prod_{i=1}^k I^{k-1}\longrightarrow\prod_{i=1}^k X_i.
\end{equation}
She then identifies a subspace, described below in \refS{thespacedeltak}, of $\prod_{i=1}^k I^{k-1}=I^{k(k-1)}$ which maps to the fat wedge $\cup_{i<j}\prod_{l\neq i,j}X_l\times (X_i\vee X_j)$, and the map $T_k$ in \refT{JohnsonThesis} sends a point in the domain of $T_k'$ to the induced map of quotients.


\subsubsection{The space $\Delta_k$}\label{S:thespacedeltak}


Represent a point in $I^{k(k-1)}$ as a $k\times k$ matrix $[t_{ij}]_{1\leq i,j\leq k}$, where $t_{ii}=0$ for all $i$ and the $i^{th}$ row are coordinates of the $i^{th}$ copy of $I^{k-1}$ in the product in \refE{tkprime}. Consider the composite
$$
I^{k(k-1)}\stackrel{T_k'}{\longrightarrow}\prod_{i=1}^k X_i\stackrel{q}{\longrightarrow}\wedge_{i=1}^k X_i
$$
where $q$ is the quotient map. 

We make the following definitions.
\begin{defin}
\begin{itemize}
\item Let $Z=\{[t_{ij}]\, |\, t_{ij}=1 \mbox{ for some } i,j\}$.
\item For $i<j$, let $W_{ij}=\{[t_{ij}]\, |\, \mbox{row $i$}=\mbox{row $j$}\}$.
\item Let $\Delta_k=I^{k(k-1)}/Z\cup \bigcup_{i<j} W_{ij}$.
\end{itemize}
\end{defin}

Johnson shows that $q\circ T_k'$ carries $Z$ and $\cup_{i<j} W_{ij}$ to the basepoint, thus inducing the map $T_k$. Unfortunately, the homotopy type of $\Delta_k$ is not easy to understand from this definition. It does, however, have a much smaller homotopy equivalent subspace whose homotopy type is relatively simple.

\begin{defin}
\begin{itemize}
\item Let $I^{k-1}_1=\{[t_{ij}]\in I^{k(k-1)} | t_{ij}=0 \mbox{ for all }j\neq 1\}$. This is a $(k-1)$-dimensional subspace of $I^{(k-1)k}$ represented in coordinates by the first column of $[t_{ij}]$.
\item Let $\widetilde{Z}=Z\cap I^{k-1}_1$.
\item Let $\widetilde{W_{ij}}=W_{ij}\cap I^{k-1}_1$.
\item Let $\widetilde{\Delta_k}=I^{k-1}_1/\widetilde{Z}\cup \bigcup_{i<j} \widetilde{W_{ij}}$.
\end{itemize}
\end{defin}

\begin{prop}\cite[Propositions 5.6 and 5.8]{J:DHT}
The inclusion $I^{k-1}_1\to I^{k(k-1)}$ induces an equivalence $\widetilde{\Delta}_k\simeq\Delta_k$. Moreover, $\widetilde{\Delta}_k$ is homotopy equivalent to $\vee_{(k-1)!}S^{k-1}$.
\end{prop}

It will be useful to describe the equivalence $\widetilde{\Delta}_k\simeq\vee_{(k-1)!}S^{k-1}$ more explicitly. Following \cite{J:DHT}, let $\Sigma_k$ denote the of bijections $\gamma:\underline{k}\to\underline{k}$ such that $\gamma(1)=1$. This can clearly be identified with a subgroup isomorphic with the symmetric group $\Sigma_{k-1}$. 
\begin{defin}\label{D:hg}
For $\gamma\in \Sigma_{k-1}$, define maps $h_\gamma:I^{k-1}\to I^{k(k-1)}$ by
$$
(s_1,\ldots,s_{k-1})\mapsto[t_{ij}]
$$
where $t_{j1}(s_1,\ldots, s_{k-1})=\max\{s_l\,|\,l<\gamma^{-1}(j)\}$ and $t_{ji}=0$ for $i\neq 1$.
\end{defin}
The image of $h_\gamma$ in $I^{k(k-1)}$ is the $(k-1)$-cell where $t_{\gamma(2)1}<t_{\gamma(3)1}<\cdots<t_{\gamma(k)1}$. Note that the image of each $h_\gamma$ is in the preimage of $\widetilde{\Delta}_k$ by the quotient map $q:I^{k(k-1)}\to\Delta_k$. Moreover, $h_\gamma$ carries the subspace $L=\{(s_1,\ldots, s_{k-1} | s_i=0\mbox{ or }1\mbox{ for some $i$ or } s_i\geq s_j\mbox{ for some $i>j$}\}$ to $Z\cup \bigcup_{i<j}W_{ij}$. The quotient $I^{k-1}/L$ is equivalent to $S^{k-1}$ since $L$ is equivalent to $\del I^{k-1}$. 
\begin{defin}\label{D:lambdag}
Define $\lambda_\gamma=q\circ h_\gamma$.
\end{defin}
The map $\vee_{\gamma\in \Sigma_{k-1}}\lambda_\gamma:\vee_{\gamma\in \Sigma_{k-1}}S^{k-1}\to\widetilde{\Delta}_k$ is a homotopy equivalence.

\subsection{Commutators, Whitehead products, and $T_k$}

\begin{defin}
Let $X$ and $Y$ be based spaces. $C:\Omega X\times \Omega Y\to \Omega (X\vee Y)$ is the \emph{commutator}, defined by
$$
C(\alpha,\beta)=\alpha\beta\alpha^{-1}\beta^{-1}.
$$
\end{defin}

Let $e$ denote the constant loop at the basepoint, and let $\alpha, \beta$ be as above. Johnson \cite{J:DHT} defines homotopies $A,B:I^2\to I$ such that $\alpha(A(s,t))$ is a homotopy from $C(\alpha,e)$ to $e$, and $\beta(B(s,t))$ a homotopy from $C(e,\beta)$ to $e$.

\begin{defin}
Let $X$ be a based space. Define $\ell:X\to\Omega\Sigma X$ by $\ell(x)=(t\mapsto t\wedge x)$, the adjoint to the identity map $\Sigma X\to\Sigma X$.  
\end{defin}

Note that if $x=\ast$ is the basepoint, the loop $\ell(\ast)$ is the constant loop $e$ at the basepoint.

\begin{defin}
Let $X_1,X_2$ be based spaces, and define
$$
C_\iota:X_1\times X_2\longrightarrow \Omega\Sigma(X_1\vee X_2)
$$
by
$C_\iota(x_1,x_2)=C(\ell(x_1),\ell(x_2))$
\end{defin}

We denote the adjoint of $C_\iota$ by $\widetilde{C}_\iota:\Sigma (X_1\times X_2)\to\Sigma(X_1\vee X_2)$. This adjoint map induces the \emph{generalized Whitehead product} (\cite[Definition 2.2]{Ark:GWP}; also see \cite[Definition 6.2]{Spe:Hil-Mil} for a relative version), according to the following proposition.

\begin{prop}\label{P:genwhitprod}
$\widetilde{C}_\iota:\Sigma (X_1\times X_2)\to\Sigma(X_1\vee X_2)$ induces a map
$$
\iota:\Sigma (X_1\wedge X_2)\longrightarrow \Sigma (X_1\vee X_2).
$$
which is the generalized Whitehead product of the inclusion maps of the $X_i$ into $X_1\vee X_2$ defined in \cite{Ark:GWP}.
\end{prop}

\begin{proof}
Note that $C_\iota(\ell(\ast),\ell(x_2))=\ell(x_2)\ell^{-1}(x_2)$ and $C_\iota(\ell(x_1),\ell(\ast))=\ell(x_1)\ell^{-1}(x_1)$, have a chosen null-homotopy given by the maps $A$ and $B$ mentioned above (we have not defined these; see \cite{J:DHT} for details). Thus the null-homotopies of the restricted maps $C_\iota|_{\ast\times X_2}$ and $C_\iota|_{X_1\times \ast}$ give a null-homotopy of the restriction $C_{\iota}|_{X_1\vee X_2}$, which in turn gives a null-homotopy of the restriction of the adjoint map $\widetilde{C}_\iota|_{\Sigma (X_1\vee X_2)}:\Sigma(X_1\vee X_2)\to\Sigma(X_1\times X_2)$. This induces the right vertical map in the following cofiber sequences
$$\xymatrix{
\Sigma(X_1\vee X_2)\ar[r]\ar[d] & \Sigma(X_1\times X_2) \ar[r]\ar[d]^{\widetilde{C}_\iota} & \Sigma(X_1\wedge X_2)\ar[d]\\
\ast \ar[r] & \Sigma(X_1\vee X_2)\ar[r] & \Sigma(X_1\vee X_2)
}
$$
By \cite[Definition 2.2]{Ark:GWP}, the induced map of cofibers $\Sigma(X_1\wedge X_2)\to\Sigma(X_1\vee X_2)$ is the generalized Whitehead product of the inclusion maps $X_1,X_2\to X_1\vee X_2$.
\end{proof}

\begin{rem}
The generalized Whitehead products can be iterated. Suppose $X_1,\ldots, X_k$ are spaces and $\gamma\in\Sigma_k$. This determines a bracketing $[\gamma^{-1}(1),\cdots[\gamma^{-1}(k-2),[\gamma^{-1}(k-1),\gamma^{-1}(k)]]\cdots]$ and hence a map
$$
\iota_\gamma:\Sigma(\wedge_{i=1}^k X_i)\longrightarrow \Sigma(\vee_{i=1}^kX_i)
$$
given by
$$
\iota_\gamma(t,x_1,\ldots, x_k)=(\iota(t,x_{\gamma^{-1}(1)},\ldots\iota(t,x_{\gamma^{-1}(k-2)},\iota(t,x_{\gamma^{-1}(k-1)},x_{\gamma^{-1}(k)})\cdots).
$$
Note that if $X_i=S^{q_i}$ for all $i$, this gives us a map $\iota_\gamma:S^{|q|-k+1}\to\vee_{i=1}^kS^{q_i}$, where $|q|=\sum_iq_i$.
\end{rem}

With this in mind, consider the iterated commutators defined below.

\begin{defin}\label{D:commutator}
Let $X_1,\ldots, X_k$ be a based spaces, and let $\gamma\in\Sigma_k$ be some permutation of $\{1,\ldots, k\}$. Define 
$$
\widehat{C}_\gamma:\prod_{i=1}^kX_i\longrightarrow\Omega\Sigma \vee_{i=1}^{k}X_i
$$
by $\widehat{C}_\gamma(x_1,\ldots, x_k)=C(\ell(x_{\gamma^{-1}(1)}),\dots, C(\ell(x_{\gamma^{-1}(k-2)},(C(\ell(x_{\gamma^{-1}(k-1)},\ell(x_{\gamma^{-1}(k)}))\cdots)$.
\end{defin}

For all $1\leq i\leq k$, one can check that the restriction of $\widehat{C}_\gamma$ to the subspace of all $(x_1,\ldots,x_k)$ such that $x_i$ is the basepoint is null-homotopic. Hence we may regard $\widehat{C}_\gamma$ as a map
\begin{equation}\label{E:koswhitprod}
\widehat{C}_\gamma:\prod_{i=1}^kX_i\longrightarrow\hofiber\left(\Omega\Sigma\vee_{i=1}^kX_i\longrightarrow\prod_{i=1}^k\Omega\Sigma\vee_{j\neq i}X_j\right).
\end{equation}
In fact, one can do much better. Johnson \cite{J:DHT} defines a map
\begin{equation*}\label{E:commutator}
C_\gamma:\prod_{i=1}^k X_i\longrightarrow\Omega\tfiber\left(S\mapsto\vee_{\underline{k}-S}\Sigma X_i\right),
\end{equation*}
which has the property that the composition
$$
\prod_{i=1}^k X_i\stackrel{C_\gamma}{\longrightarrow}\Omega\tfiber\left(S\mapsto\vee_{\underline{k}-S}\Sigma X_i\right)\stackrel{proj}{\longrightarrow}\Omega\Sigma\vee_{\underline{k}}X_i
$$
is the map $\widehat{C}_\gamma$ in \refD{commutator}. In analogy with \refP{genwhitprod}, we have the following lemma.

\begin{lemma}\label{L:mapofcofibers}
$C_\gamma$ induces a map $D_\gamma:\wedge_{i=1}^k X_i\to \Omega\tfiber\left(S\mapsto\vee_{\underline{k}-S}\Sigma X_i\right)$. 
\end{lemma}

\begin{proof}
It is sufficient to establish this in the case where $\gamma=\iota$ is the identity map. Consider the functor $R\mapsto \prod_{i\in R}X_i$, where we regard $\prod_{i\in R}X_i\subset \prod_{i=1}^k X_i$ as the subspace of tuples $(x_1,\ldots, x_k)$ such that $x_j=\ast$ if $j\notin R$. From the definition, it is clear that the restriction of $C_\iota: \prod_{i=1}^k X_i$ to $\prod_{i\in R} X_i$ maps to $\Omega\tfiber\left(S\mapsto\vee_{\underline{k}\cap R-S\cap R}\Sigma X_i\right)$. We therefore have a map of diagrams
$$
\left(R\mapsto \prod_{i\in R}X_i\right)\longrightarrow\left(R\mapsto \Omega\tfiber\left(S\mapsto\vee_{\underline{k}\cap R-S\cap R}\Sigma X_i\right)\right)
$$
which induces a map of total homotopy cofibers (\cite[Definition 1.4]{CalcII})
$$
\tcofiber\left(R\mapsto \prod_{i\in R}X_i\right)\longrightarrow\tcofiber\left(R\mapsto \Omega\tfiber\left(S\mapsto\vee_{\underline{k}\cap R-S\cap R}\Sigma X_i\right)\right).
$$
It is clear that $\tcofiber\left(R\mapsto \prod_{i\in R}X_i\right)$ is equivalent to $\wedge_{i=1}^k X_i$. If $R$ is a proper subset of $\underline{k}$, the space $\Omega\tfiber\left(S\mapsto\vee_{\underline{k}\cap R-S\cap R}\Sigma X_i\right)$ is contractible because two faces of the cube in question are identical. Therefore $\tcofiber\left(R\mapsto \Omega\tfiber\left(S\mapsto\vee_{\underline{k}\cap R-S\cap R}\Sigma X_i\right)\right)\simeq \Omega\tfiber\left(S\mapsto\vee_{\underline{k}-S}\Sigma X_i\right)$, and so we have an induced map
$$
\wedge_{i=1}^k X_i\longrightarrow \Omega\tfiber\left(S\mapsto\vee_{\underline{k}-S}\Sigma X_i\right).
$$
\end{proof}

We will also refer to the adjoint $\widetilde{D}_\gamma$ of $D_\gamma$ as a generalized Whitehead product. Recall the set $\Sigma_{k-1}$ of all bijections $\gamma:\underline{k}\to\underline{k}$ such that $\gamma(1)=1$. 

\begin{defin}[Definition 6.7 of \cite{J:DHT}]\label{D:gammagh}
For $\gamma,\gamma'\in \Sigma_{k-1}$, $\Gamma_{\gamma\gamma'}:S^k\to S^k$ is the map which makes the diagram below commutative
$$
\xymatrix{
\prod_{i=1}^k X_i \ar[rr]^-{\Omega T_k\circ C_\gamma}\ar[d]_q & & \Map_\ast(\Sigma\widetilde{\Delta_k},\Sigma X_1\wedge\cdots\wedge \Sigma X_k)\ar[d]^{\Lambda_{\gamma'}^\ast}\\
\wedge_{i=1}^k X_i \ar[rr]_-{x\mapsto\Gamma_{\gamma\gamma'}\wedge x} & & \Omega^k\Sigma^k(\wedge_{i=1}^k X_i)
}
$$
\end{defin}

The map $\Lambda_{\gamma'}$ is the suspension of a map $\lambda_{\gamma'}:S^{k-1}\to \widetilde{\Delta}_k$ which gives rise to the equivalence $\lambda=\vee_{\gamma\in \Sigma_{k-1}}\lambda_\gamma:\vee_{\gamma\in \Sigma_{k-1}}S^{k-1}\to\widetilde{\Delta}_k$, and $\Lambda_{\gamma'}^\ast$ is the map induced by $\Lambda_{\gamma'}$.

Johnson proves the following.

\begin{prop}[Proposition 6.8 of \cite{J:DHT}]
$\Gamma_{\gamma\gamma'}:S^k\to S^k$ has degree one if $\gamma=\gamma'$ and is otherwise null-homotopic.
\end{prop}

Finally, Johnson defines a map
\begin{equation}\label{E:themapgamma}
\Gamma:\vee_{\gamma\in \Sigma_{k-1}}\wedge_{i=1}^k X_i\longrightarrow \prod_{\gamma\in \Sigma_{k-1}}\Omega^k\Sigma^k \wedge_{i=1}^k X_i
\end{equation}
which can be thought of as a $(k-1)!\times (k-1)!$ matrix made up of the $\Gamma_{\gamma\gamma'}$. $\Gamma$ fits into the following commutative diagram
$$
\xymatrix{
\vee_{\gamma\in \Sigma_{k-1}}\prod_{i=1}^k X_i \ar[rr]^-{\vee_{\gamma}\Omega T_k\circ C_\gamma}\ar[d]_q & & \Map_\ast\left(\Sigma\widetilde{\Delta_k},\Sigma X_1\wedge\cdots\wedge \Sigma X_k\right)\ar[d]^{\Lambda}\\
\vee_{\gamma\in \Sigma_{k-1}}\wedge_{i=1}^k X_i \ar[rr]_-{\Gamma} & & \prod_{\gamma\in \Sigma_{k-1}}\Omega^k\Sigma^k\left(\wedge_{i=1}^k X_i\right).
}
$$
Here $\Lambda$ is the map induced by $\vee_{\gamma'\in \Sigma_{k-1}}\Lambda_{\gamma'}$.


\section{Koschorke's generalized $\mu$-invariants}\label{S:MilnorInvts}


In this section we review and generalize work of Koschorke \cite{Kosch:Milnor} on higher Hopf invariants and describe its relationship with the homotopy theoretic results of the previous section. In particular we will use ideas of Koschorke to make geometric sense of the map $H_k$ in \refE{mainthmeqn}. He describes a geometric interpretation of ``reduced homotopy groups'' of a wedge of spheres $\vee_i S^{q_i}$ in terms of links and higher-order linking numbers, which he shows are a generalization of Milnor's $\mu$-invariants \cite{Mil:LinkGroups}.


\subsection{Higher Hopf invariants and total homotopy fibers}


\begin{defin}\label{D:reducedhtpy}\cite[Equation 14]{Kosch:Milnor}
The \emph{reduced homotopy groups} of a wedge of pointed spaces $X_1,\ldots, X_k$ are defined by
$$
\tilde{\pi}_\ast\left(\vee_{i=1}^{k}X_i\right):=\cap_{i=1}^{k}\ker\left(\pi_\ast\left(\vee_{i=1}^{k}X_i\right)\longrightarrow\pi_\ast\left(X_1\vee \cdots\vee X_{i-1}\vee X_{i+1}\vee\cdots\vee X_k\right)\right).
$$
\end{defin}

Recall the iterated Whitehead products from \refP{genwhitprod} and the remark following. One of the main results of \cite{Kosch:Milnor} is the following theorem, which tries to say that the iterated Whitehead products and the ``higher Hopf invariants'' $h_\gamma$ (see \cite{Kosch:Milnor} for the definition) are dual. Let $|q|=\sum q_i$.

\begin{thm}\label{T:kosthm1}[\cite{Kosch:Milnor}, Theorem 3.1]
We have a commutative diagram 

$$\xymatrix{
& \tilde{\pi}_\ast\left(\vee_{i=1}^{k}S^{q_i}\right)\ar[ddr]^{h_\gamma} &   \\
 & & \\
\pi_\ast\left(S^{\abs{q}-k+1}\right) \ar[uur]^{\left(\iota_{\gamma'}\right)_\ast}\ar[rr]_{\pm\delta_{\gamma\gamma'}\Sigma^{\infty}}& & \pi_{\ast-\abs{q}+k-1}^S\\
}
$$

for all permutations $\gamma,\gamma'\in\Sigma_{k-1}$. In addition, if $q_i\geq\max\{2,p-\abs{q}+k+1\}$ for $i=1,\ldots, k$, then $h:=\oplus_{\gamma\in\Sigma_{k-1}}h_\gamma:\tilde{\pi}_p\left(\vee_{i=1}^{k}S^{q_i}\right)\rightarrow\oplus_{(k-1)!}\pi_{p-\abs{q}+k-1}^S$ is an isomorphism. Here $\delta_{\gamma\gamma'}$ is the Kronecker symbol.
\end{thm}

Koschorke \cite{Kosch:Milnor} also gives two closely related geometric interpretations of the map $h_\gamma$: (1) it measures an ``overcrossing locus'' of a bordism class of a link with $k-1$ components, and (2) it measures iterated intersections of manifolds which bound the aforementioned $k-1$ component link. We refer the reader to \cite[Section 3]{Kosch:Milnor} for details.

Our goal is to make a space-level version of \refT{kosthm1}, where $\tfiber\left(S\mapsto\vee_{i\in \underline{k}-S}S^{q_i}\right)$ replaces $\tilde{\pi}_\ast\left(\vee_{i=1}^{k}S^{q_i}\right)$, $S^{\abs{q}+k-1}$ replaces $\pi_\ast S^{\abs{q}+k-1}$, and $QS^{\abs{q}+k-1}$ replaces $\pi_{\ast-\abs{q}+k-1}^S$, and state the conclusion as the connectivity of a certain map. Note, however, that the homotopy groups of the total homotopy fiber above are not the reduced homotopy groups. We can do even better by replacing the spheres with the suspension of an arbitrary space.

Let $X_1,\ldots, X_k$ be based spaces, and let $D=\vee_\gamma D_\gamma:\vee_\gamma \wedge_{i=1}^k X_i\to \Omega\tfiber(S\mapsto \vee_{\underline{k}-S}\Sigma X_i)$ be the map given by \refL{mapofcofibers}, and let $\widetilde{D}$ denote its adjoint. Let $\widetilde{\Gamma}$ denote the adjoint to the map $\Gamma$ from \refE{themapgamma}. What follows is our generalization of \refT{kosthm1}. Our conclusion differs from that of \refT{kosthm1} somewhat: we already know the connectivity of $T_k$, and what we conclude is the connectivity of the ``total generalized Whitehead product'' $\widetilde{D}$.

\begin{thm}\label{T:kosthm1gengen}
Let $X_1,\ldots, X_k$ be $n$-connected based spaces. Then the diagram
$$\xymatrix{
& \tfiber(S\mapsto\vee_{i\in \underline{k}-S}\Sigma X_i)\ar[ddr]^{T_k} &   \\
 & & \\
\vee_{(k-1)!}\Sigma\wedge_{i=1}^kX_i \ar[uur]^{\widetilde{D}}\ar[rr]_{\widetilde{\Gamma}}& & \prod_{(k-1)!}\Omega^{k-1}\Sigma^k(X_1\wedge\cdots\wedge X_k)\\
}
$$
commutes, and $T_k$ is $(k+1)(n+1)$-connected, $\widetilde{\Gamma}$ is $(2kn+2)$-connected, and thus $\widetilde{D}$ is $((k+1)(n+1)-1)$-connected.
\end{thm}

\begin{proof}
Consider the commutative diagram
$$\xymatrix{
 & & \Omega\tfiber(S\mapsto\vee_{\underline{k}-S}\Sigma X_i)\ar[d]^{\Omega T_k}\\
\vee_{g\in G_k}\wedge_{i=1}^k X_i \ar[rru]^D\ar[rrd]_\Gamma & & \Omega\Map_\ast(\widetilde{\Delta}_k,\Sigma X_1\wedge\cdots\wedge\Sigma X_k)\ar[d]^{\simeq}\\
& & \prod_\gamma \Omega^k\Sigma^k(X_1\wedge\cdots\wedge X_k)
}
$$
By \refT{JohnsonThesis}, $\Omega T_k$ is $((k+1)(n+1)-1)$-connected, and Johnson shows that $\Gamma$ is $(2kn+1)$-connected. It follows that $D$ is $((k+1)(n+1)-2)$-connected, since $(2kn+1)\geq (k+1)(n+1)-1$ for all $n$ if $k\geq 1$. \refP{connadjoints} implies that the map
$$
\widetilde{D}:\vee_{g\in G_k}\Sigma\wedge_{i=1}^kX_i\longrightarrow \tfiber(S\mapsto\vee_{\underline{k}-S}\Sigma X_i)
$$
is $((k+1)(n+1)-1)$-connected, since $\wedge_{i=1}^kX_i$ is $kn$-connected, and $2kn+1\geq (k+1)(n+1)-2$ for all $n$ if $k\geq 1$. Similarly, \refP{connadjoints} also implies that the map $\widetilde{\Gamma}$ is $(2kn+2)$-connected. Thus, in the commutative diagram
$$\xymatrix{
& \tfiber(S\mapsto\vee_{i\in \underline{k}-S}\Sigma X_i)\ar[ddr]^{T_k} &   \\
 & & \\
\vee_{(k-1)!}\Sigma\wedge_{i=1}^kX_i \ar[uur]^{\widetilde{D}}\ar[rr]_{T_k\circ \widetilde{D}=\widetilde{\Gamma}}& & \prod_{(k-1)!}\Omega^{k-1}\Sigma^{k-1}(\Sigma(X_1\wedge\cdots\wedge X_k))\\
}
$$
$\widetilde{D}$ is $((k+1)(n+1)-1)$-connected, $T_k$ is $(k+1)(n+1)$-connected, and $\widetilde{\Gamma}$ is $(2kn+2)$-connected.
\end{proof}

The map $\widetilde{D}$ is the generalized Whitehead product, and so for the following we set $\iota_g=\widetilde{D}_g$ and $\iota=\widetilde{D}$ to use more standard notation. For comparison with \refT{kosthm1}, we have

\begin{cor}\label{C:kosthm1gen}
There is a commutative (up to homotopy) diagram of spaces

$$\xymatrix{
& \tfiber(S\mapsto\vee_{i\in \underline{k}-S}S^{q_i})\ar[ddr]^{T^s_k} &   \\
 & & \\
\vee_{(k-1)!}S^{|q|-k+1} \ar[uur]^{\iota}\ar[rr]_{Q}& & \prod_{(k-1)!}QS^{|q|-k+1}\\
}
$$

where $T^s_k$ is $(|q|-k+\min\{q_i\})$-connected, $Q$ is $(2(|q|-k)+1)$-connected, and thus $\iota$ is $(|q|-k+\min\{q_i\}-1)$-connected.
\end{cor}

\begin{proof}
The properties of $\Gamma$ (see \refD{gammagh} and \refE{themapgamma}) ensure that the composite of $\Gamma$ with the canonical map $\Omega^{k-1}\Sigma^{k-1}S^{|q|-k+1}\to QS^{|q|-k+1}$ is homotopic to the canonical map $S^{|q|-k+1}\to QS^{|q|-k+1}$.
\end{proof}


In the special case where $q_i=n-1$ for all $i$, which is our main interest, we have the following.

\begin{cor}\label{C:kosthm1gen2}
There is a commutative (up to homotopy) diagram of spaces

$$\xymatrix{
& \tfiber(S\mapsto\vee_{\underline{k}-S}S^{n-1})\ar[ddr]^{T^s_k} &   \\
 & & \\
\vee_{(k-1)!}S^{k(n-2)+1} \ar[uur]^{\iota}\ar[rr]_{Q}& & \prod_{(k-1)!}QS^{k(n-2)+1}\\
}
$$

where $T^s_k$ is $(k+1)(n-2)$-connected, $Q$ is $(2k(n-2)+1)$-connected, and thus $\iota$ is $((k+1)(n-2)-1)$-connected.
\end{cor}

Our next goal is to give a geometric interpretation of the map $T^s_k$ in \refC{kosthm1gen2} in terms of higher-order linking numbers. This will essentially follow from a space-level version of the Pontryagin-Thom construction, and so we pause to discuss the necessary details of the construction of a space of cobordisms.

%

\section{Cobordism spaces}\label{S:LinkBordism}


We begin with a very brief description of cobordism spaces. The author has used these in \cite{M:Emb} and \cite{M:LinkNumber}. We review only the most basic details here.

\begin{defin}[Simplicial Model for a Cobordism Space]\label{D:cobspacedefn} Let $\xi,\eta$ be vector bundles over a space $X$ and let $d=\dim(\eta)-\dim(\xi)$. The simplicial set $C_{\bullet}^{\xi-\eta}(X)$ has as its $k$-simplices the set $C_k=\{(W^{d+k},f,\phi)\}$ where $W$ is a smooth $(k+d)$-dimensional manifold embedded in $\R^{\infty}\times \Delta^k$, $W$ is transverse to $\R^{\infty}\times \del_S\Delta^k$ for all nonempty subsets $S\subset\{0,1,\ldots, k\}$, $f:W\rightarrow X$ is continuous and proper, and $\phi: TW\oplus f^*(\xi)\rightarrow f^*(\eta)$ is a stable isomorphism.
\end{defin}

\begin{rem}\label{cobspacethomspace}
The cobordism space $C_\bullet^{\xi-\eta}(X)$ is equivalent to $QT(X;\xi-\eta)$. To see the equivalence, consider the subcomplex of the total singular complex of $QT(X;\xi-\eta)$ consisting of those $k$-simplices $\kappa:\Delta^k\rightarrow \Omega^n\Sigma^n(T(X;\xi-\eta))$ that correspond to maps $\kappa':\Sigma^n(\Delta^k)\rightarrow \Sigma^n(T(X;\xi-\eta))$ which are transverse to the zero section of $T(X;\xi-\eta)$. This sub-complex is equivalent to the full complex and the map $\kappa\mapsto \kappa'^{-1}(0)$ to the cobordism model is an equivalence. See \cite{CalcI}.

To define $QT(X;\xi-\eta)$ precisely, choose a vector bundle monomorphism $\eta\to\epsilon^i$, and define $QT(X;\xi-\eta)=\Omega^iQT(X;\xi\oplus\epsilon^i/\eta)$. For $i$ large, the homotopy type of $\Omega^iQT(X;\xi\oplus\epsilon^i/\eta)$ is independent of the monomorphism and $i$. See \cite{gkw} for the same construction. To see the functoriality of $QT(X;\xi-\eta)$, note that if $U\subset X$ is open, then we define $QT(U;\xi-\eta)$ by pulling back the bundle $\xi\oplus\epsilon^i/\eta$ along the inclusion map $U\to X$ and forming the Thom space. That is, if $U\subset V$ is an inclusion, then we have a commutative diagram
$$\xymatrix{
S(U;\xi\oplus\epsilon^i/\eta)\ar[d]\ar[r] & D(U;\xi\oplus\epsilon^i/\eta)\ar[d]\\
S(V;\xi\oplus\epsilon^i/\eta)\ar[r] & D(V;\xi\oplus\epsilon^i/\eta)\\
}
$$
in which the vertical arrows are given by the pullback. This induces a map of horizontal cofibers $T(U;\xi\oplus\epsilon^i/\eta)\to T(V;\xi\oplus\epsilon^i/\eta)$.
\end{rem}

Now suppose $P$ is a smooth closed manifold. Our cobordism spaces can be viewed as contravariant functors from $\mathcal{O}(P)$, the poset of open subsets of $P$, to $\Top$ by composing with the realization functor. They are contravariant because the map from the manifold to $P$ is required to be proper.

\begin{prop}\label{P:cobordtubular}\cite[Proposition 21]{M:Emb}
Suppose $P$ is a smooth closed manifold and $U\subset P$ is a tubular neighborhood of a submanifold $S$. That is, $U$, is a $k$-disk bundle over $S$. Then there is an equivalence $$C_\bullet^{\xi-\eta}(U)\rightarrow C_{\bullet-k}^{\xi\oplus\nu(S\subset U)-\eta}(S).$$
\end{prop}


\subsection{Cobordism interpretation of $\Map_\ast(P_+,Q S^n)$}


\begin{prop}\label{P:bordinterp}
Let $P$ be a smooth closed manifold of dimension $p$. There is an equivalence, natural in $U\in\mathcal{O}(P)$,
$$
\Map_\ast(U_+,Q S^n)\longrightarrow Q_+T(U,\epsilon^n-TU).
$$
\end{prop}

\begin{proof}
We begin by defining the map in question. Without loss of generality, suppose $U=P$. Let $f:\Delta^j\to\Map_\ast\left(P_+,Q S^n\right)$ be a $j$-simplex. From this we produce a $j$-simplex in $C_\bullet^{\epsilon^n-TP}(P)$ as follows. Regard $f:\Delta^j\longrightarrow \Map_\ast\left(P_+,Q S^n\right)$ as a map $f:\Delta^j\times P_+\longrightarrow Q S^n$. By compactness of $\Delta^j\times P_+$, this is a map $f:\Delta^j\times P_+\longrightarrow \Omega^m\Sigma^m S^n$ for some $m$, which determines a map $\widehat{f}:S^m\wedge(\Delta^j\times P)_+\longrightarrow S^m\wedge S^n$.

By a small homotopy of $\widehat{f}$, we may assume it is transverse to $0\times 0$, where $0$ is not the wedge point in either copy of the sphere (we always use $\infty$ for this). Consider the manifold $W=\widehat{f}^{-1}(0\times 0)$, where  We have
\begin{enumerate}
\item $W=\widehat{f}^{-1}(0\times 0)$ is a smooth submanifold of $S^m\times \Delta^j\times P$ of codimension $m+n$. That is, $\dim(W)=j+p-n$. Moreover, $W$ is transverse to the faces of $\Delta^j$.
\item There is a proper map $W\to P$ given by the projection of the inclusion of $W$ in $S^m\times \Delta^j\times P$.
\item Transversality gives an isomorphism $TW\oplus \epsilon^m\oplus\epsilon^n\cong\epsilon^m\oplus\epsilon^j\oplus TP$.
\end{enumerate}

Thus, a $k$-simplex in $\Map_\ast\left(P_+,Q S^n\right)$ determines a $k$-simplex in $C^{\epsilon^n-TP}_\bullet(P)$.

The map $\Map_\ast\left(U_+,Q S^n\right)\to C_\bullet^{\epsilon^n-TU}(U)$ defined above is a natural transformation of good functors which are polynomial of degree $\leq1$ (see \cite[Definition 1.1, Definition 2.1]{W:EI1}). By \cite[Theorem 5.1]{W:EI1}, it is enough to check this map is an equivalence on open sets $U$ diffeomorphic to a disk. We have, then, the following sequence of equivalences.
\begin{eqnarray*}
\Map_\ast\left(D^p_+,Q S^n\right)&\simeq&Q S^n\\
&\simeq& QT\left(\ast;\epsilon^n\right)\\
&\simeq& QT\left(D^p;\epsilon^n-\epsilon^p\right)\\
&\simeq& C_\bullet^{\epsilon^n-TD^p}\left(D^p\right)
\end{eqnarray*}
The penultimate equivalence follows from \refP{cobordtubular}, and the last equivalence follows from Remark \ref{cobspacethomspace}.
\end{proof}

We also obtain a stable range description of $\Map_\ast\left(P_+,S^n\right)$.

\begin{cor}
The canonical map $S^n\to QS^n$ is $(2n+1)$-connected and hence induces a $(2n+1-p)$-connected map
$$
\Map_\ast\left(P_+,S^n\right)\longrightarrow Q_+T\left(P,e^n-TP\right).
$$
\end{cor}

Combining \refP{bordinterp}, \refC{stableunloopedtk}, and Remark \ref{cobspacethomspace}, we have our main result.

\begin{thm}\label{T:maintheoremrestate}
Let $P$ be a smooth manifold of dimension $p$. There is a $((k+1)(n-2)-p)$-connected map
\begin{equation*}
\Map_\ast\left(P_+,\tfiber\left(S\mapsto \vee_{\underline{k}-S} S^{n-1}\right)\right)\rightarrow \prod_{(k-1)!}Q_+T\left(P; \epsilon^{k(n-2)+1}-TP\right).
\end{equation*}
\end{thm}

What remains to be done is to interpret this map geometrically. The details are implicit in the proof of \refP{bordinterp}.


\subsection{$T^s_k$ as an overcrossing locus}\label{S:InvtsMflds}


Following Koschorke \cite{Kosch:Milnor}, we will give a geometric description of the map $T_k^s$ as an ``overcrossing locus'' of a link. It is similar to the way one computes the linking number of a classical link by planar projection.

Recall that
\begin{equation}\label{E:Tkstableeqn}
\left(T_k^s\right)_\ast:\Map_\ast\left(P_+,\tfiber\left(S\mapsto\vee_{\underline{k}-S}\Sigma S^{n-2}\right)\right)\longrightarrow\prod_{(k-1)!}\Map_\ast\left(P_+,QS^{k(n-2)+1}\right)
\end{equation}
is the ``stabilization'' of the map
$$
\left(T_k\right)_\ast:\Map_\ast\left(P_+,\tfiber\left(S\mapsto\vee_{\underline{k}-S}\Sigma S^{n-2}\right)\right)\longrightarrow\prod_{(k-1)!}\Map_\ast\left(P_+,\Omega^{k-1}\Sigma^{k-1}S^{k(n-2)+1}\right)
$$
induced by $T_k$. It is built from the map 
$$
\left(T_k'\right)_\ast:\Map_\ast\left(P_+,\tfiber\left(S\mapsto\vee_{\underline{k}-S}\Sigma S^{n-2}\right)\right)\longrightarrow\Map_\ast\left(P_+,\Map\left(I^{k(k-1)},\prod_{i=1}^k\Sigma S^{n-2}\right)\right).
$$
We view $\Map_\ast\left(P_+,\tfiber\left(S\mapsto\vee_{\underline{k}-S}\Sigma S^{n-2}\right)\right)$ as a parametrized total homotopy fiber, and so we use the same notation as in \refD{tfiber}. Let $\Phi\in\Map_\ast\left(P_+,\tfiber\left(S\mapsto\vee_{\underline{k}-S}\Sigma S^{n-2}\right)\right)$. This gives us maps $\Phi_i:P_+\times I^{k-1}\to \Sigma S^{n-2}$ for $i=1$ to $k$. $\left(T_k'\right)_\ast(\Phi)$ is the map
$$
P_+\times \prod_{i=1}^kI^{k-1}\longrightarrow\prod_{i=1}^k\Sigma S^{n-2}
$$
given by
$$
\left(p,\vec{t}_1,\ldots,\vec{t}_k\right)\mapsto\left(\Phi_1\left(p,\vec{t}_1\right),\ldots,\Phi_k\left(p,\vec{t}_k\right)\right)
$$
where $\vec{t}_i$ is the $i^{th}$ row of the matrix $[t_{ij}]$ using notation as in \refS{thespacedeltak}, giving coordinates for the $i^{th}$ copy of $I^{k-1}$. Recalling the equivalence $\Delta_k\simeq\widetilde{\Delta}_k$, consider the restriction to
\begin{equation}\label{E:restrictedmap}
P_+\times I^{k-1}_1\longrightarrow\prod_{i=1}^k\Sigma S^{n-2}
\end{equation}
and note that
$$
\left(p,0,t_{21},\ldots,t_{k1}\right)\mapsto\left(\Phi_1\left(p,0\right),\Phi_2\left(p,t_{21}\right),\ldots,\Phi_k\left(p,t_{k1}\right)\right).
$$
Let $x_i\in \Sigma S^{n-2}$ be other than the wedge point in the $i^{th}$ copy of $\Sigma S^{n-2}$. By a small homotopy we may assume the $\Phi_i$ are transverse to $x_i$ for all $i$. Consider the manifolds
$$
L_{1}=\Phi_1^{-1}(x_1)
$$
and
$$
L_{i}=\Phi_i^{-1}(x_i)
$$
for $2\leq i\leq k$. The locus of points $L_g$ consisting of $p\in P_+$ such that $p$ is the the image of the projection of $L_i$ to $P_+$ for all $i$ and $t_{g(2)1}<t_{g(3)1}<\cdots<t_{g(k)1}$ is precisely the manifold measured by the composition in \refE{Tkstableeqn} according to the bordism interpretation of the target in \refP{bordinterp}. Note that $\del L_1=\emptyset$ and $\del L_i\subset P\times\{0\}$. In fact, $L=L_1\cup \bigcup_{2\leq i\leq k}\del L_i$ is a $k$-component link in $P$, and it can easily be described as the preimage of $x_1\coprod x_2\coprod\cdots\coprod x_k\in\vee_{\underline{k}}\Sigma S^{n-2}$ by the map $P_+\to \vee_{\underline{k}}\Sigma S^{n-2}$. Thus we can describe $L_g$, for each $g$, as an \emph{overcrossing locus} of a bordism of a $k$-component link.

\begin{rem}
The description of $L$ as an overcrossing locus is reminiscent of Goodwillie's proof of \cite[Lemma 2.7]{CalcII}. This lemma is the key ingredient in the inductive step of his proof of the higher Blakers-Massey Theorem, which otherwise follows formally from facts about cubical diagrams. In fact, he constructs a link in the same manner we do, and the fact that the overcrossing locus is empty through a range of dimensions translates into the connectivity estimate in the statement of the higher Blakers-Massey Theorem. Our construction can therefore be taken as a measure of the failure of this theorem, and it identifies, in a range, the first few non-trivial homotopy groups in terms of bordism.
\end{rem}

%

\section{Multivariable manifold calculus and generalizations of Milnor's invariants}\label{S:Applications}


Here we make explicit the connection with Koschorke's work on link maps. Up to this point we have only invoked the homotopy theoretic parts of his work, and now we will show how the work in the previous sections is closely related to the study of link maps. A generalization of Milnor's invariants arise naturally in certain multi-linear homogeneous layers of the multivariable Taylor Tower for link maps.


\subsection{The $\mu$-invariants of link maps}\label{S:muinvariants}


We begin by returning to the work of Koschorke \cite{Kosch:Milnor}, and recall how \refT{kosthm1} plays a role in describing analogs of Milnor's $\mu$-invariants \cite{Mil:LinkGroups} for link maps. Let $P_1,\ldots, P_{k+1}$ be smooth closed manifolds with $\dim\left(P_i\right)=p_i$.

\begin{defin}
Let $\Link\left(P_1,\ldots, P_{k+1};\R^n\right)$ be the space of smooth maps $f=\coprod_if_i:\coprod_iP_i\to \R^n$ such that $f_i\left(P_i\right)\cap f_j\left(P_j\right)=\emptyset$ for all $i\neq j$. It is topologized as a subspace of the space of smooth maps $\prod_i\Map\left(P_i,\R^n\right)$. The number of \emph{components} of $f$ is $k+1$. We say two link maps are \emph{link-homotopic} if there is a path between them in the space of link maps. The \emph{trivial link} or \emph{unlink} is the link map $u=\coprod_i u_i$ such that $u_i$ is constant for all $i$ (for some choice of constants). A link map $f$ is \emph{almost trivial} if the restriction of $f$ to each \emph{sub-link} $\coprod_{i\in S} S^{p_i}$, where $S$ is a proper subset of $\{1,\ldots, k+1\}$, is link-homotopic to the trivial link.
\end{defin}

The philosophy of the $\mu$-invariants in \cite{Mil:LinkGroups} and \cite{Kosch:Milnor} is that one can define invariants of ``order'' $|S|$ for each nonempty subset $S$ of $\{1,\ldots, k+1\}$ provided that all invariants of order $|R|<k+1$ for $R\subset S$ vanish. Thus, inductively, one is interested in defining invariants of order $k+1$ for a $(k+1)$-component link. One criterion that ensures all invariants of order less than $k+1$ vanish is that the link is almost trivial. One then attempts to produce invariants from the induced map $\widehat{f}=\prod_{i=1}^{k+1}f_i:P_1\times\cdots\times P_{k+1}\to C(k+1,\R^n)$, where the target is the configuration space of $k+1$ points in $\R^n$. Our generalization in \refS{MapModels} will involve homotopy limits of diagrams of spaces of maps of products into configuration spaces.

Koschorke \cite{Kosch:Milnor} restricts attention to spheres, so let $P_i=S^{p_i}$ for the time being, and he endows his spheres with a basepoint $\ast\in S^{p_i}$ for all $i$. He calls a link map $f$ \emph{$\kappa$-Brunnian} if the restriction of $\widehat{f}$ to $S^{p_1}\times\cdots\times S^{p_{i-1}}\times\ast\times S^{p_{i+1}}\times\cdots\times S^{p_{k+1}}$ is null-homotopic for all $i=1$ to $k+1$. This is clearly stronger than saying a link map is almost trivial ($\kappa$-Brunnian implies almost triviality). Let $\abs{p}=\sum p_i$. If $f$ is $\kappa$-Brunnian, Koschorke \cite{Kosch:Milnor} shows that this implies there is a unique element $\widehat{\kappa}(f)\in\widetilde{\pi}_{\abs{p}}\vee_{k}S^{n-1}$ (the reduced homotopy group from \refD{reducedhtpy}) determined by $f$. He then can apply the map $h_\gamma$ of \refT{kosthm1} to such a class. He shows \cite[Corollary 6.2]{Kosch:Milnor} that in the classical case where $p_i=1$ and $n=3$, a link is $\kappa$-Brunnian if and only if it is almost trivial, and moreover, that the $\mu$-invariants $h_\gamma(\widehat{\kappa}(f))$ are, up to a fixed sign, the same as those defined by Milnor in \cite{Mil:LinkGroups}.

Our generalization builds on Koschorke's ideas in the manner indicated in the introduction. It has been constructed from the point of view of a multivariable manifold calculus of functors, and so this requires a brief discussion of ``manifold calculus'' (due to Weiss and Goodwillie \cite{W:EI1,GW:EI2}) and a multivariable generalization of it (due to the author and Voli\'{c} \cite{MV:Multi}). Calculus provides a natural organizational framework for these higher-order invariants.


\subsection{Manifold calculus}\label{S:ManifoldCalc}


Let $P$ be a smooth closed manifold of dimension $p$, and $\mathcal{O}(P)$ the poset of open subsets of $P$. Manifold calculus is concerned with the study contravariant functors $F:\mathcal{O}(P)\rightarrow\Top$ which are ``good'' (Definition 1.1 of \cite{W:EI1}). One approximates a good functor $F$ with a sequence of functors $T_kF$, $k\geq 0$ which are ``polynomial of degree $\leq k$''. They form a tower $\cdots\to T_kF\to T_{k-1}F\to\cdots\to T_0F$ analogous to the Taylor series of a smooth function, with $T_kF$ the analog of the $k^{th}$ degree Taylor polynomial. A functor $F$ is polynomial of degree $\leq 0$ if it is essentially constant, polynomial of degree $\leq 1$ if it satisfies excision, and polynomial of degree $\leq k$ if it satisfies $k^{th}$-order excision (see Definition 2.2 of \cite{W:EI1} for details).

\begin{example}
Let $P$ be a smooth manifold and $U\in\mathcal{O}(P)$. The functor $U\mapsto\Map(U,Z)$ is polynomial of degree $\leq 1$ for any space $Z$, and for a smooth manifold $N$, so is $U\mapsto\Imm(U,N)$, the space of smooth immersions. More generally, if $p:Z\rightarrow P$ is a fibration with fiber $F$, and we denote by $\Gamma(P,F)$ the space of sections, then $U\mapsto\Gamma(U,F)$ is polynomial of degree $\leq 1$.
\end{example}


\subsection{Multivariable manifold calculus}\label{S:MultiCalc}


The author and Voli\'{c} \cite{MV:Multi} have developed a multivariable version of manifold calculus, motivated by the space of link maps. Throughout the rest of this section, we will work with smooth closed manifolds $P_1,\ldots, P_{k+1}$ of dimension $p_1,\ldots, p_{k+1}$.  Let $\mathcal{O}\left(\coprod_iP_i\right)$ be the poset of open subsets of the disjoint union of the $P_i$.

The space of link maps $\Link\left(P_1,\ldots, P_{k+1};N\right)$ defines a contravariant functor on ${O}\left(\coprod_iP_i\right)$, since an inclusion $U\subset V$ gives rise to a restriction $\Link(V;N)\to\Link(U;N)$. We can view this as a functor of several variables as follows. There is an isomorphism of categories $\mathcal{O}\left(\coprod_iP_i\right)\cong \prod_i\mathcal{O}\left(P_i\right)$ given by the map $U\mapsto \vec{U}=\left(U_1,\ldots, U_{k+1}\right)$, where $U_i=U\cap P_i$. We write $\mathcal{O}\left(\vec{P}\right)$ in place of $\prod_i\mathcal{O}\left(P_i\right)$. Multivariable manifold studies contravariant functors $F:\mathcal{O}\left(\vec{P}\right)\rightarrow\Top$ satisfying axioms analogous to those of \cite[Definition 1.1]{W:EI1}, and seeks to approximate them with multivariable polynomial functors $T_{\vec{\jmath}}F$, where $\vec{\jmath}=(j_1,\ldots, j_{k+1})$ is a multi-index. There is a multi-tower of functors consisting of the $T_{\vec{\jmath}}F$ with maps $T_{\vec{\jmath}}F\to T_{\vec{\imath}}F$ whenever $\vec{\jmath}\geq\vec{\imath}$ (the partial ordering $\geq$ is determined entry-wise on multi-indices of integers $\vec{\jmath}$ and $\vec{\imath}$). A polynomial of degree $\vec{\jmath}=\left(j_1,\ldots, j_{k+1}\right)$ is one which is polynomial of degree $j_i$ in the $i$th variable.

\begin{example}
For any space $X$, the functor $(U_1,\ldots, U_{k+1})\mapsto\prod_i\Map(U_i,X)$ is a polynomial of degree $\leq 1$ in each variable, and hence a polynomial of degree $\leq\vec{1}=(1,1,\ldots, 1)$.
\end{example}

For a smooth manifold $P$ we let $\mathcal{O}_l(P)$ denote the full subcategory of $\mathcal{O}(P)$ whose objects are the open sets diffeomorphic to at most $l$ open balls. For a multi-index $\vec{\jmath}=(j_1,\ldots, j_{k+1})$ of non-negative integers we put $\mathcal{O}_{\vec{\jmath}}\left(\vec{P}\right)=\prod_i\mathcal{O}_{j_i}\left(P_i\right)$. We define the multivariable polynomial approximations to a $F$ functor as follows.

\begin{defin}\label{D:jMultiStage}  Define the \emph{$\vec{\jmath}^{th}$ degree Taylor approximation of $F$} to be
$$T_{\vec{\jmath}}F\left(\vec{U}\right)=\underset{\mathcal{O}_{\vec{\jmath}}\left(\vec{U}\right)}{\holim}\, \, F.$$
\end{defin}


\subsection{Mapping space models}\label{S:MapModels}


Let $\vec{\delta}=(\delta_1,\delta_2,\ldots, \delta_{k+1})$. Our goal is to give some models for $T_{\vec{\delta}}\Link(P_1,\ldots, P_{k+1};N)$ when $\delta_i$ is equal to $0$ or $1$ for all $i$.  These are the approximations which contain the information about our generalization of Koschorke's invariants.

\begin{example}
Define $\Lambda_{(1,1)}(P_1,P_2;N)$ to be the homotopy limit of the following diagram

\begin{equation}\label{E:Lambda11}
\xymatrix{
  \Map(P_1\times P_2,N)  & \Map(P_1\times P_2,C(2,N))  \ar[l]\ar[r] & \Map(P_1\times P_2,N) \\
  \Map(P_1,N) \ar[u]  & & \Map(P_2,N)\ar[u]
  }
\end{equation}

We claim that
$$
T_{(1,1)}\Link(P_1,P_2;N)\simeq\Lambda_{(1,1)}(P_1,P_2;N).
$$
There is clearly a natural transformation of functors
$$
\Link(P_1,P_2;N)\longrightarrow\Lambda_{(1,1)}(P_1,P_2;N),
$$
and it is enough by \cite[Theorem 4.14]{MV:Multi} to check that it is an equivalence in the cases where the $P_i$ are either disks or empty, since $\Lambda_{(1,1)}$ is polynomial of degree $\leq (1,1)$. If $P_1$ is a disk and $P_2=\emptyset$, the diagram in \refE{Lambda11} reduces to $\Map(P_1,N)\simeq\Link(P_1,\emptyset;N)$. A similar argument holds if $P_1=\emptyset$ and $P_2$ is a disk. If both $P_1$ and $P_2$ are disks, then $\Link(P_1,P_2;N)\simeq C(2,N)$, and $\Lambda_{(1,1)}(P_1,P_2;N)$ is equivalent to
$$
\holim\left(
\raisebox{.75cm}{
\xymatrix{
 N  & C(2,N)  \ar[l]\ar[r] & N \\
  N \ar[u]  & & N\ar[u]
  }}
\right).
$$
Since the vertical arrows above are equivalences, the homotopy limit of the diagram above is in turn equivalent to $\holim(N\leftarrow C(2,N)\rightarrow N)$, which is itself equivalent to $C(2,N)$.
\end{example}

More generally we can make mapping space models for the stages $T_{\vec{\delta}}\Link(P_1,\ldots, P_{k+1};N)$, where $\vec{\delta}=(\delta_1,\ldots, \delta_{k+1})$, and each $\delta_i$ is either $0$ or $1$, following \cite{gkw}. This can be done for any $\vec{\delta}$, but we restrict attention to the case $\vec{\delta}\leq\vec{1}=(1,1,\ldots, 1)$ for the reasons described above. This will also allow us to avoid some complications beyond the scope of this work.

Let $\vec{\underline{\delta}}=(\underline{\delta}_1,\ldots, \underline{\delta}_{k+1})$ be such that $\underline{\delta}_i$ is either $\underline{0}=\emptyset$ or $\underline{1}=\{1\}$. Consider the poset of all pairs $(\vec{S},\vec{R})$ such that $\vec{R}\subset\vec{S}\subset\vec{\underline{\delta}}$. The poset structure is given by $(\vec{S}_1,\vec{R}_1)\leq(\vec{S}_2,\vec{R}_2)$ if $\vec{S}_1\subset\vec{S}_2$ and $\vec{R}_2\subset\vec{R}_1$.

Let $\vec{P}^{\vec{S}}=\prod P_i^{S_i}$, and consider the functor $\left(\vec{S},\vec{R}\right)\mapsto\Map\left(\vec{P}^{\vec{S}},C\left(|\vec{R}|,N\right)\right)$.

\begin{defin}\label{D:mappingspacemodel}
Define 
$$\Lambda_{\vec{\delta}}\left(\vec{P};N\right)=\underset{\left(\vec{S},\vec{R}\right)\neq\left(\vec{\emptyset},\vec{\emptyset}\right)}{\holim}\Map\left(\vec{P}^{\vec{S}},C\left(|\vec{R}|,N\right)\right).$$
\end{defin}

\begin{prop}\label{P:mappingspacemodel}
$\Lambda_{\vec{\delta}}\left(\vec{P};N\right)$ is polynomial of degree $\leq\vec{\delta}$, and $\Link\left(\vec{U};N\right)\simeq\Lambda_{\vec{\delta}}\left(\vec{U};N\right)$ for $\vec{U}\in\mathcal{O}_{\vec{\delta}}\left(\vec{P}\right)$. Therefore $\Lambda_{\vec{\delta}}\left(\vec{P};N\right)\simeq T_{\vec{\delta}}\Link\left(\vec{P};N\right)$.
\end{prop}

\begin{proof}
That $\Lambda_{\vec{\delta}}$ is polynomial of degree $\leq\vec{\delta}$ follows from the fact that it is the homotopy limit of a diagram of polynomial functors of degree $\leq\vec{\delta}$ (see \cite[Example 2.5]{W:EI1}). By inspection, its values agree with $\Link\left(\vec{U};N\right)$ when $\vec{U}\in\mathcal{O}_{\vec{\delta}}\left(\vec{P}\right)$.
\end{proof}


\subsection{Multivariable homogeneous functors}


\begin{defin}\label{D:MultiHomogeneous}
A functor $E:\mathcal{O}\left(\vec{P}\right)\rightarrow\Top$ is \emph{homogeneous of degree $\vec{\jmath}$} if it is polynomial of degree $\leq \vec{\jmath}$ and $\holim_{\vec{k}<\vec{\jmath}}T_{\vec{k}}E\left(\vec{U}\right)$ is contractible for all $\vec{U}$.
\end{defin}

\begin{defin}\label{D:MultiLayer}
We define the \emph{$\vec{\jmath}^{th}$ layer of the Taylor multi-tower} of $F$ to be the functor

$$L_{\vec{\jmath}}F=\hofiber\left(T_{\vec{\jmath}}F\rightarrow \underset{\vec{k}<\vec{\jmath}}{\holim}\, T_{\vec{k}}F\right).$$
\end{defin}

In order for this to make sense, we need to choose a basepoint, so we assume that we have a preferred element of $F\left(\vec{P}\right)$ chosen to base all spaces in sight. The above definition is justified by the fact that $L_{\vec{\jmath}}F$ is indeed homogeneous of degree $\vec{\jmath}$.

\begin{example}
Let $e=(e_1,e_2)\in\Link(P_1,P_2;N)$ be the basepoint. By definition, $L_{(1,1)}\Link(P_1,P_2;N)$ is the total homotopy fiber of the square
$$\xymatrix{
\Lambda_{(1,1)}\left(P_1,P_2;N\right) \ar[d]\ar[r]& \Lambda_{(0,1)}\left(P_1,P_2;N\right)\ar[d]\\
\Lambda_{(1,0)}\left(P_1,P_2;N\right)\ar[r] & \Lambda_{(0,0)}\left(P_1,P_2;N\right).
}
$$
The equivalences $\Lambda_{(1,0)}(P_1,P_2;N)\simeq\Map(P_1,N)$, $\Lambda_{(0,1)}(P_1,P_2;N)\simeq\Map(P_2,N)$, and $\Lambda_{(0,0)}(P_1,P_2;N)\simeq\ast$ imply that $\Lambda_{(1,1)}\left(P_1,P_2;N\right)$ is equivalent to
$$
\holim\left(\Map\left(P_1\times P_2,N\right)\rightarrow\Map\left(P_1\times P_2,C\left(2,N\right)\right)\leftarrow\Map\left(P_1\times P_2,N\right)\right).
$$
Then by inspection,
$$
\Lambda_{(1,1)}\left(P_1,P_2;N\right)\simeq\Map\left(P_1\times P_2,\tfiber\left(R\mapsto C\left(\underline{2}-R,N\right)\right)\right).
$$
\end{example}

\begin{rem}
The author has shown \cite{M:LinkNumber} that there is a map $\hofiber(\Link(P_1,P_2;N)\to\Map(P_1,N)\times\Map(P_2,N))$ to a cobordism space which deserves to be called the generalized linking number. The author and Goodwillie \cite{GM:LinksEstimates} have shown that this map is highly connected. This generalized linking number factors through $\Lambda_{(1,1)}(P_1,P_2;N)$. 
\end{rem}

The example above suggests a more general result when there are more than two components. The classification of multivariable homogeneous functors, \cite[Theorem 5.18]{MV:Multi}, together with \refP{mappingspacemodel} above, implies the following lemma.

\begin{lemma}\label{L:linklayer}
Let $\vec{1}=(1,1,\ldots, 1)$. For smooth closed manifolds $P_1,\ldots, P_{k+1}$, write $\vec{P}=\left(P_1,\ldots, P_{k+1}\right)$. The natural map
$$
L_{\vec{1}}\Link\left(\vec{P};N\right)\longrightarrow \Map\left(P_1\times\cdots\times P_{k+1},\tfiber\left(R\mapsto C\left({\underline{k+1}-R},N\right)\right)\right).
$$
is an equivalence of functors of $\vec{P}$.
\end{lemma}

\begin{proof}
Write $\mathcal{C}_{k+1}(N)$ in place of $\tfiber(R\mapsto C({\underline{k+1}-R},N)$ for brevity. On the one hand, Theorem 5.18 of \cite{MV:Multi} implies that 
$$
L_{\vec{1}}\Link\left(\vec{P};N\right)\simeq \Gamma\left(P_1\times\cdots\times P_{k+1},\mathcal{C}_{k+1}\left(N\right)\right),
$$
where $\Gamma$ stands for the section space of some fibration $p:Z\to P_1\times\cdots\times P_{k+1}$ whose fibers are $\mathcal{C}_{k+1}(N)$. On the other hand, sections are maps, there is a natural transformation of polynomials of degree $\leq\vec{1}$ given by inclusion
$$
\Gamma\left(P_1\times\cdots\times P_{k+1},\mathcal{C}_{k+1}\left(N\right)\right)\longrightarrow\Map\left(P_1\times\cdots\times P_{k+1},\mathcal{C}_{k+1}\left(N\right)\right).
$$
When $\vec{U}=(U_1,\ldots, U_{k+1})\in\mathcal{O}_{\vec{1}}(\vec{P})$, it is clear that the values of $\Gamma\left(U_1\times\cdots\times U_{k+1},\mathcal{C}_{k+1}\left(N\right)\right)$ and $\Map\left(U_1\times\cdots\times U_{k+1},\mathcal{C}_{k+1}\left(N\right)\right)$ agree, since a fibration over a disk becomes trivial, and thus by Theorem 4.14 of \cite{MV:Multi} the map is an equivalence.
\end{proof}

We are interested in the special case where $N=\R^n$. We now examine the diagram $R\mapsto C({\underline{k+1}-R},\R^n)$ in order to reduce ourselves to the case of \refT{maintheorem}. One easy observation is the following:

\begin{lemma}\label{L:configtosphere}
The total homotopy fiber of the $(k+1)$-cube of based spaces
$$
R\mapsto C\left({\underline{k+1}-R},\R^n\right)
$$
is equivalent to the total homotopy fiber of the $k$-cube
$$
S\mapsto \vee_{\underline{k}-S}S^{n-1}.
$$
\end{lemma}

\begin{proof}
Write $R\mapsto C({\underline{k+1}-R},\R^n)$ as a map of $k$-cubes
$$
\left(S\mapsto C\left(\underline{k}\cup\{k+1\}-S,\R^n\right)\right)\longrightarrow \left(S\mapsto C\left(\underline{k}-S,\R^n\right)\right),
$$
where $S$ ranges over subsets of $\underline{k}=\{1,\ldots, k\}\subset\underline{k+1}$.
For all $S$, the restriction map
$$
C\left(\underline{k}\cup\{k+1\}-S,\R^n\right)\longrightarrow C\left(\underline{k}-S,\R^n\right)
$$
is a fibration with fiber $\R^n-\{k-|S|\mbox{ points}\}\simeq\vee_{\underline{k}-S}S^{n-1}$. We may identify the wedge point with the image of the $(k+1)^{st}$ point under this equivalence.
\end{proof}

Putting \refL{linklayer} and \refL{configtosphere} together yields the following theorem, and demonstrates how the domain of the map in \refT{maintheorem} arises naturally when studying the space of link maps from a functor calculus point of view.

\begin{thm}\label{T:multilinklayer}
The natural map
$$
L_{\vec{1}}\Link\left(\vec{P};N\right)\longrightarrow \Map\left(P_1\times\cdots\times P_{k+1},\tfiber\left(S\mapsto \vee_{\underline{k}-S}S^{n-1}\right)\right).
$$
is an equivalence of functors of $\vec{P}$.
\end{thm}


\subsection{Generalizations of Milnor's invariants for link maps}\label{S:genmilnor}


In this section we generalize Koschorke's generalization of Milnor's invariants to arbitrary manifolds linking in Euclidean space. Let $e,f\in\Link\left(P_1,\ldots, P_{k+1};\R^n\right)$ be link maps, where $e$ is the basepoint. The link maps $e$ and $f$ determine elements of $T_{\vec{\delta}}\Link\left(P_1,\ldots, P_{k+1};\R^n\right)$, and therefore they also determine elements $t_{\vec{1}}e,t_{\vec{1}}f\in \holim_{\vec{\delta}<\vec{1}} T_{\vec{\delta}}\Link\left(P_1,\ldots, P_{k+1};\R^n\right)$. Assume there is given a path between $t_{\vec{1}}e$ and $t_{\vec{1}}f$ in $\holim_{\vec{\delta}<\vec{1}} T_{\vec{\delta}}\Link\left(P_1,\ldots, P_{k+1};\R^n\right)$. Then $f$ determines an element of $L_{\vec{1}}\Link\left(\vec{P};N\right)$. The assumption that there is a path between the elements $t_{\vec{1}}e$ and $t_{\vec{1}}f$ is analogous to (a relative version of) almost triviality of a link map. This is admittedly a little opaque, and to better see how it is related to the classical notion, first consider the following definition.

\begin{defin}\label{D:relativealmosttrivial}
Suppose $e$ and $f$ are given as above. For $S\subset \underline{k+1}$ write $\vec{P}_S=\coprod_{i\notin S}P_i$, and consider the $(k+1)$-cube $S\mapsto\Link\left(\vec{P}_S;N\right)$. We say that $e$ and $f$ are \emph{relatively almost trivial} if the elements of $\holim_{S\neq\emptyset}\Link\left(\vec{P}_S;N\right)$ determined by $e$ and $f$ are homotopic.
\end{defin}

For example, given two $3$-component link maps $e,f:P_1\coprod P_2\coprod P_3\to N^n$, they are relatively almost trivial if the restrictions of $e$ and $f$ to $P_1\coprod P_2$, $P_1\coprod P_3$, and $P_2\coprod P_3$ are homotopic, and that the restriction of these homotopies to $P_1$, $P_2$, and $P_3$ are homotopic where it makes sense to compare them. If two link maps $e$ and $f$ are relatively almost trivial, then the homotopies above determine a path between $t_{\vec{1}}f$ and $t_{\vec{1}}e$. In fact, a condition weaker than \refD{relativealmosttrivial} will suffice to give such a path. Looking back at \refD{mappingspacemodel}, in order for $t_{\vec{1}}e$ and $t_{\vec{1}}f$ to be homotopic, this requires that the corresponding elements of $\Map\left(\vec{P}^{\vec{S}},C\left(|\vec{R}|,N\right)\right)$ determined by $e=(e_1,\ldots, e_{k+1})$ and $f=(f_1,\ldots, f_{k+1})$ (which are the products of some of the components of $e$ and $f$) are homotopic, and that these homotopies are homotopic where it makes sense to compare them.  Certainly the induced maps of configuration spaces are homotopic whenever the link maps are, but the converse need not hold.

\begin{prop}\label{P:linklayersection}
Let $t_{\vec{1}}e,t_{\vec{1}}f\in \Map\left(P_1\times\cdots\times P_{k+1},\tfiber\left(R\mapsto C({\underline{k+1}-R},\R^n\right)\right)$, and suppose they are relatively almost trivial. Then $t_{\vec{1}}e$ and $t_{\vec{1}}f$ define elements $s_e, s_f$ of $\Map\left(P_1\times\cdots\times P_{k+1},\tfiber\left(S\mapsto \vee_{\underline{k}-S}S^{n-1}\right)\right)$ such that
\begin{enumerate}
\item The sections $s_e$ and $s_f$ are unique up to homotopy, and
\item We can arrange for $s_e$ to be the constant section whose value is the wedge point for all $p=(p_1,\ldots, p_{k+1})\in P_1\times\cdots\times P_{k+1}$.
\end{enumerate}
\end{prop}

\begin{proof}
Regard $t_{\vec{1}}e$ as the basepoint of $\Map\left(P_1\times\cdots\times P_{k+1},\tfiber\left(R\mapsto C({\underline{k+1}-R},\R^n\right)\right)$. Thus for each $p=(p_1,\ldots, p_{k+1})\in P_1\times\cdots\times P_{k+1}$ we have a preferred basepoint of $C(\underline{k+1},\R^n)$ given by the configuration $(e(p_1),\ldots, e(p_{k+1}))$. We may arrange for the fiberwise (over $P_1\times\cdots\times P_{k+1}$) equivalence $C(\{k+1\},\R^n-\{e(p_i),i\neq k+1\}\simeq\vee_{\underline{k}}S^{n-1}$ given in \refL{configtosphere} to identify $e(p_{k+1})$ with the wedge point.

To show that $s_f$ is unique up to homotopy is straightforward. 
\end{proof}

That $s_f$ is unique up to homotopy assures us that invariants we extract from $s_f$ really are invariants of $f$ itself.


\section{Acknowledgments}\label{S:Thanks}

The author would like to thank Greg Arone, Vladimir Chernov, Michael Ching, Tom Goodwillie, Mike Hopkins, Brenda Johnson, and Ismar Voli\'{c} for helpful conversations. He would also like to thank Harvard University and Wellesley College for their hospitality.


\def\cprime{$'$} \def\cprime{$'$}
\providecommand{\bysame}{\leavevmode\hbox to3em{\hrulefill}\thinspace}
\providecommand{\MR}{\relax\ifhmode\unskip\space\fi MR }
\providecommand{\MRhref}[2]{%
  \href{http://www.ams.org/mathscinet-getitem?mr=#1}{#2}
}
\providecommand{\href}[2]{#2}

%

\begin{thebibliography}{10}

\bibitem{Ark:GWP}
Martin Arkowitz, \emph{The generalized {W}hitehead product}, Pacific J. Math
  (1962), 7--23.

\bibitem{ES:BM}
Graham~J. Ellis and Richard Steiner, \emph{Higher-dimensional crossed modules
  and the homotopy groups of $(n+1)$-ads}, J. Pure Appl. Algebra \textbf{46}
  (1987), 117--136.

\bibitem{CalcI}
Thomas~G. Goodwillie, \emph{Calculus {I}: {T}he first derivative of
  pseudoisotopy theory}, $K$-Theory \textbf{4} (1990), 1--27.

\bibitem{CalcII}
\bysame, \emph{Calculus {II}: {A}nalytic functors}, $K$-Theory \textbf{5}
  (1991/92), no.~4, 295--332.

\bibitem{gkw}
Thomas~G. Goodwillie, John~R. Klein, and Michael~S. Weiss, \emph{A {H}aefliger
  style description of the embedding calculus tower}, Topology \textbf{42}
  (2003), no.~3, 509--524.

\bibitem{GM:LinksEstimates}
Thomas~G. Goodwillie and Brian~A. Munson, \emph{A stable range description of
  the space of link maps}, Algebr. Geom. Topol. \textbf{10} (2010), 1305--1315.

\bibitem{GW:EI2}
Thomas~G. Goodwillie and Michael Weiss, \emph{Embeddings from the point of view
  of immersion theory {II}}, Geom. Topol. \textbf{3} (1999), 103--118
  (electronic).

\bibitem{J:DHT}
Brenda Johnson, \emph{The {D}erivatives of {H}omotopy {T}heory}, Trans. Amer.
  Math. Soc. \textbf{347} (1995), no.~4, 1295--1321.

\bibitem{Kosch:Milnor}
Ulrich Koschorke, \emph{A generalization of {M}ilnor's {$\mu$}-invariants to
  higher-dimensional link maps}, Topology \textbf{36} (1997), no.~2, 301--324.

\bibitem{Mil:LinkGroups}
John Milnor, \emph{Link groups}, Ann. of Math. (2) \textbf{59} (1954),
  177--195.

\bibitem{M:Emb}
Brian~A. Munson, \emph{Embeddings in the {$3/4$} range}, Topology \textbf{44}
  (2005), no.~6, 1133--1157.

\bibitem{M:LinkNumber}
\bysame, \emph{A manifold calculus approach to link maps and the linking
  number}, Algebr. Geom. Topol. \textbf{8} (2008), no.~4, 2323--2353.

\bibitem{MV:Multi}
Brian~A. Munson and Ismar Voli\'c, \emph{Multivariable manifold calculus of
  functors}, to appear in Forum Mathematicum.

\bibitem{Spe:Hil-Mil}
Christopher Spencer, \emph{The {M}ilnor-{H}ilton theorem and {W}hitehead
  products}, J. London Math. Soc. \textbf{2} (1971), no.~4, 291--303.

\bibitem{W:EI1}
Michael Weiss, \emph{Embeddings from the point of view of immersion theory
  {I}}, Geom. Topol. \textbf{3} (1999), 67--101 (electronic).

\end{thebibliography}

\end{document}